\documentclass[12pt, twoside]{article}

\usepackage{amsmath,amsthm,amssymb}

\usepackage{times}

\usepackage{amssymb, amsfonts, amsbsy, latexsym, epsfig, color}

\usepackage{enumerate}

\makeatletter

\def\author#1{\gdef\autrun{\def\and{\unskip, }#1}\gdef\@author{#1}}

\makeatother

 \newtheorem{theorem}{\bf Theorem}
 \newtheorem{lemma}[theorem]{\bf Lemma}

\begin{document}

\author{Simeon Ball and Bence Csajb\'ok\footnote{29 November 2017. The first author acknowledges the support of the project MTM2014-54745-P of the Spanish {\em Ministerio de Econom\'ia y Competitividad}. The second author is supported by the J\'anos Bolyai Research Scholarship of the Hungarian Academy of Sciences. The second author acknowledges the support of OTKA Grant No. K 124950.}}

\title{On sets of points with few odd secants }

\date{}

\maketitle

\begin{abstract}
We prove that, for $q$ odd, a set of $q+2$ points in the projective plane over the field with $q$ elements has at least $2q-c$ odd secants, where $c$ is a constant and an odd secant is a line incident with an odd number of points of the set.
\end{abstract}

\section{Introduction}

Let $\mathrm{PG}(2,q)$ denote the projective plane over ${\mathbb F}_q$, the finite field with $q$ elements. An odd secant to a set $S$ of points of $\mathrm{PG}(2,q)$ is a line of  $\mathrm{PG}(2,q)$ which is incident with an odd number of points of $S$. 
In \cite{BBFT2014}, P. Balister, B. Bollob{\' a}s, Z. F{\" u}redi and J. Thompson study sets of points which minimise the number of odd secants for a given cardinality. Among other things, they prove that if $|S| \leqslant q+1$ then $o(S)$, the number of odd secants, is at least $|S|(q+2-|S|)$ and that equality is achieved if and only if $S$ is chosen to be an arc, a set of points in which no three are collinear. They also provide a general lower bound for sets of size $q+2$ proving that if $q$ is odd and $|S|=q+2$ then $o(S) \geqslant \frac{3}{2}(q+1)$. This lower bound was subsequently improved to $o(S) \geqslant\frac{1}{5}(8q+12)$ for $q > 13$ in \cite[Theorem 6.17]{Csajbok2017}.

Conjecture 11 from \cite{BBFT2014} maintains that if $q$ is odd and $|S|=q+2$ then $o(S) \geqslant 2q-2$. Observe that a conic, together with an external point (a point incident with two tangents to the conic) is a set of $q+2$ points with $2q-2$ odd secants.
We will prove the following theorem, which proves Conjecture 11 from  \cite{BBFT2014}, up to a constant.

\begin{theorem} \label{main}
Let $S$ be a set of $q+2$ points in $\mathrm{PG}(2,q)$. If $q$ is odd then there is a constant $c$ such that $o(S) \geqslant 2q-c$.
\end{theorem}

In \cite[Theorem 3.2]{Vandendriessche2015}, Vandendriessche classifies those point sets $S$ which satisfy $|S| + o(S) \leqslant 2q$, for $q$ odd, apart from one possible example which was subsequently excluded in \cite[Proposition 6.19]{Csajbok2017}.

The number of secants of a set of $q+2$ points was first studied in \cite{BB1989} by A. Blokhuis and A. A. Bruen and refined later in \cite{seminuclear} by A. Blokhuis. The finite field Kakeya problem in the plane is to determine the minimal size of a Besicovich set, that is, an affine point set containing a line in every direction. By duality, this problem is equivalent to asking for the smallest number of lines meeting a set of $q+2$ points where one of the points (the point corresponding to the line at infinity) is incident only with bi-secants. Such points of a $(q+2)$-set are called internal nuclei, they were first studied by A. Bichara and G. Korchm\'aros in \cite{BK1982}, where Lemma~\ref{szero} was proven. In \cite{kakeya}, A.~Blokhuis and F.~Mazzocca proved that in $\mathrm{PG}(2,q)$, $q$ odd, the dual of a Besicovich set of minimal size is an irreducible conic together with an external point, the same configuration that we mentioned before. In the case $q$ is even there exist hyperovals in $\mathrm{PG}(2,q)$, i.e. arcs of size $(q+2)$. It is straightforward to see that hyperovals do not have odd secants. For stability results on point sets with few odd secants in $\mathrm{PG}(2,q)$, $q$ even, we refer to \cite{setofeven} by T. Sz\H{o}nyi and Zs. Weiner.

B. Segre used his celebrated lemma of tangents to prove that the set of tangents to an arc of size $q+2-t$ are contained in a curve of degree $d$ in the dual plane, where $t=d$ if $q$ is even, and $d=2t$ when $q$ is odd \cite{Segre1, Segre2}. In the paper \cite{untouchable}, by A. Blokhuis, A. Seress and H. A. Wilbrink, point sets without tangents were considered and Segre's technique was applied to associate curves to the "long secants", that is, lines meeting the set without tangents in more than two points. In the present paper we apply a coordinate-free, scaled variant of Segre's original argument, developed in \cite{Ball2012} and \cite{planararcs}, to point sets which admit both tangents and long secants.

\section{Sets of $q+2$ points with few odd secants.}

Let $S$ be a set of $q+2$ points in $\mathrm{PG}(2,q)$.

For $x \in S$, let the {\em weight} of $x$ be
$$
w(x) = \sum \frac{1}{| \ell \cap S|},
$$
where the sum is over the lines $\ell$ incident with $x$ and an odd number of points of $S$.

Let $S_{0}$ be the subset of $S$ of points which are incident with $q+1$ bi-secants, so the points of $S$ of weight zero.

The following lemma is from \cite{BK1982}.

\begin{lemma} \label{szero}
If $q$ is odd then $|S_0| \leqslant 2$.
\end{lemma}

\begin{proof}
Suppose that $\{x,y,z\}$ are three points of $S_0$. Then $\{x,y,z\}$ is a basis. With respect to this basis, the line joining $x$ to $s=(s_1,s_2,s_3)$ is $\ker (s_3X_2-s_2X_3)$. As we vary the point $s \in S \setminus \{x,y,z \}$, we obtain every line $\ker (X_2+aX_3)$, for each non-zero $a \in {\mathbb F}_q$ exactly once, so we have
$$
\prod_{s \in S \setminus \{x,y,z \}} \frac{-s_2}{s_3}=-1.
$$
Similarly, from the points $y$ and $z$ we obtain,
$$
\prod_{s \in S \setminus \{x,y,z \}} \frac{-s_3}{s_1}=-1
$$
and
$$
\prod_{s \in S \setminus \{x,y,z \}} \frac{-s_1}{s_2}=-1.
$$
Multiplying these three expressions together, we have that $1=-1$.
\end{proof}

Let $S_{4/3}$ be the subset of $S$ of points which are incident with precisely one tangent and one $3$-secant and $q-1$ bi-secants, so the points of $S$ of weight $\frac{4}{3}$.

\begin{lemma} \label{notconstant}
If $|S_{4/3}|\leq c/2$ for some constant $c$, then the number of odd secants to $S$ is at least $2q-c$.
\end{lemma}

\begin{proof}
The number of odd secants is $\sum_{x \in S} w(x).$ By Lemma~\ref{szero}, $|S_0| \leqslant 2$ and for all $x \in S \setminus (S_0 \cup S_{4/3})$, the weight of $x$ is at least $2$. 
Then $\sum_{x \in S} w(x) \geq (q-c/2)2$.
\end{proof}

For each point $x \in S_{4/3}$, let $f_x$ denote the linear form whose kernel is the tangent to $S$ at $x$ and let $g_x$ denote the linear form whose kernel is the $3$-secant to $S$ at $x$.

The $3$-secants meeting $S_{4/3}$ partition $S_{4/3}$ into at least $\frac{1}{3}|S_{4/3}|$ parts, where two points are in the same partition if and only if they are joined by a $3$-secant.

Assuming that $S$ has less than $2q-c$ odd secants, for some constant $c$, Lemma~\ref{notconstant} implies that $|S_{4/3}|$ is more than a constant and so there is a subset $S'$ of $S_{4/3}$ with more than a constant number of points, with the property that if $x$ and $y$ are in $S'$ then the line joining $x$ and $y$ is not incident with any other point of $S$, i.e. where $S'$ contains at most one point from each part of the partition by $3$-secants.

Fix an element $e \in S'$ and scale $f_x$ so that $f_x(e)=f_e(x)$ for every $x\in S'$. Similarly, scale $g_x$ so that $g_x(e)=g_e(x)$. From now on, each point of the plane will be represented by a fixed vector of the underlying $3$-dimensional vector space (and not by a one-dimensional subspace).

\begin{lemma} \label{segre}
For all $x,y \in S' $,
$$
f_x(y)g_y(x)=-f_y(x)g_x(y).
$$
\end{lemma}

\begin{proof}
Fix a basis of the vector space so that with respect to the basis $x=(1,0,0)$, $y=(0,1,0) $ and $e=(0,0,1)$. The line joining $x$ to $s=(s_1,s_2,s_3)$ is $\ker (s_3X_2-s_2X_3)$. Since $x \in S_{4/3}$, as we vary the point $s \in S \setminus \{x,y,e \}$, we obtain every line $\ker (X_2+aX_3)$, for each non-zero $a \in {\mathbb F}_q$ exactly once, except the line $\ker f_x$ which does not occur, and the line $\ker g_x$ which occurs twice. Now $g_x(X)=g_x(y)X_2+g_x(e)X_3$ and $f_x(X)=f_x(y)X_2+f_x(e)X_3$, so we have
$$
\frac{g_x(y)}{g_x(e)}\frac{f_x(e)}{f_x(y)}\prod_{s \in S \setminus \{x,y,e \}} \frac{-s_2}{s_3}=-1.
$$
Similarly, from the points $y$ and $e$ we obtain,
$$
\frac{g_y(e)}{g_y(x)}\frac{f_y(x)}{f_y(e)}\prod_{s \in S \setminus \{x,y,e \}} \frac{-s_3}{s_1}=-1
$$
and
$$
\frac{g_e(x)}{g_e(y)}\frac{f_e(y)}{f_e(x)}\prod_{s \in S \setminus \{x,y,e \}} \frac{-s_1}{s_2}=-1.
$$
Multiplying these three expressions together, we have that
$$
g_x(y)f_x(e)g_y(e)f_y(x)g_e(x)f_e(y)=-g_x(e)f_x(y)g_y(x)f_y(e)g_e(y)f_e(x).
$$
The lemma now follows, since $f_e(x)=f_x(e)$, etc.
\end{proof}

For $x,y \in S'$, let
$$
b_{xy}(X)=f_x(X)g_y(X)-g_x(X)f_y(X).
$$

We will need the following two identities which easily follow from the definition of $b_{xy}(X)$. 

\begin{lemma}\label{needitlater}
	For all $x,y,z\in S'$,
\begin{equation}
\label{id1}
g_x(X)b_{yz}(X)+g_y(X)b_{zx}(X)+g_z(X)b_{xy}(X)=0.
\end{equation}
	For all $x,y,z,s \in S'$,
\begin{equation}
\label{id2}
b_{xs}(X)b_{yz}(X)+ b_{ys}(X)b_{zx}(X)+b_{zs}(X)b_{xy}(X)=0.
\end{equation} 
\end{lemma}\qed



\begin{lemma} \label{gfunc}
For all $x,y,z,w \in S'$, 
$$
\det(w,x,z)g_w(y)g_z(w)b_{yx}(w)=\det(w,x,y)g_w(z)g_y(w)b_{zx}(w).
$$
\end{lemma}

\begin{proof}
Since $f_w(X)$ is a linear form, it is determined by its value at three points. Hence,
$$
\det(x,y,z)f_w(X)=f_w(x)\det(X,y,z)+f_w(y)\det(x,X,z)+f_w(z)\det(x,y,X).
$$
Since $f_w(w)=0$,
$$
f_w(x)\det(w,y,z)+f_w(y)\det(x,w,z)+f_w(z)\det(x,y,w)=0.
$$
Eliminating $\det(w,y,z)$ from this equation by, using the similar equation for $g_w$, we get
$$
(g_w(x)f_w(y)-f_w(x)g_w(y))\det(x,w,z)+(g_w(x)f_w(z)-f_w(x)g_w(z))\det(x,y,w)=0
$$
By Lemma~\ref{segre},
$$
g_w(x)f_w(y)=g_w(y)f_w(x)\frac{f_y(w)}{g_y(w)}\frac{g_x(w)}{f_x(w)}
$$
and
$$
g_w(x)f_w(z)=g_w(z)f_w(x)\frac{f_z(w)}{g_z(w)}\frac{g_x(w)}{f_x(w)}.
$$
Combining these with the previous equation we get the required equation.
\end{proof}

\begin{lemma} \label{boxlemma}
For all $x,y,z \in S'$, $b_{xy}(z) \neq 0$.
\end{lemma}

\begin{proof}
Suppose $b_{xy}(z) = 0$. Then, applying Lemma~\ref{segre} twice,
$$
g_z(y)f_z(x)-f_z(y)g_z(x)=0.
$$
This implies that $g_z(y)f_z(X)-f_z(y)g_z(X)$ is zero at $x$, $y$ and $z$ which implies that it is identically zero. Hence, $\ker f_z=\ker g_z$, which is a contradiciton.
\end{proof}

For any homogeneous polynomial $f$ in three variables, let $V(f)$ denote the set of points of $\mathrm{PG}(2,q)$ where $f$ vanishes.

\begin{lemma} \label{curve}
Let $s,x,y,z \in S'$. If $q$ is odd then the points of $S'$ are contained in $V(\psi_{xyzs})$, where $\psi_{xyzs}$ is the polynomial of degree $6$ which, with respect to the basis $\{x,y,z\}$, is 
$$
\psi_{xyzs}(X)=s_1 b_{xs}(X)b_{yz}(X)X_2X_3+s_2 b_{ys}(X)b_{zx}(X)X_1X_3+s_3 b_{zs}(X)b_{xy}(X)X_1X_2.
$$
Also, $V(\psi_{xyzs})$ does not change under any permutation of $\{ x,y,z,s\}$.
\end{lemma}

\begin{proof}
Let $w \in S'$. By Lemma~\ref{gfunc},  
\begin{equation}
\label{eq1}
\det(w,s,z)g_w(y)g_z(w)b_{ys}(w)=\det(w,s,y)g_w(z)g_y(w)b_{zs}(w)
\end{equation}
and
\begin{equation}
\label{eq2}
\det(w,s,z)g_w(x)g_z(w)b_{xs}(w)=\det(w,s,x)g_w(z)g_x(w)b_{zs}(w).
\end{equation}
With respect to the basis $\{x,y,z\}$, since $g_w(w)=0$, we have
\begin{equation}
\label{eq3}
g_w(x)w_1+g_w(y)w_2+g_w(z)w_3=0.
\end{equation}
Our aim is to eliminate $w$ from the subscripts. By \eqref{eq1} and \eqref{eq2} we get
\[g_w(x)=\frac{\det(w,s,x)g_w(z)g_x(w)b_{zs}(w)}{\det(w,s,z)g_z(w)b_{xs}(w)}\]
and
\[g_w(y)=\frac{\det(w,s,y)g_w(z)g_y(w)b_{zs}(w)}{\det(w,s,z)g_z(w)b_{ys}(w)},\]
respectively. Substituting the above equations into \eqref{eq3} gives
\[\frac{\det(w,s,x)g_w(z)g_x(w)b_{zs}(w)}{\det(w,s,z)g_z(w)b_{xs}(w)}w_1+
\frac{\det(w,s,y)g_w(z)g_y(w)b_{zs}(w)}{\det(w,s,z)g_z(w)b_{ys}(w)}w_2+
g_w(z)w_3=0.\]
After multiplying by $b_{xs}(w)b_{ys}(w)\det(w,s,z)g_z(w)/g_w(z)$ we obtain
\[b_{ys}(w)\det(w,s,x)g_x(w)b_{zs}(w)w_1+
b_{xs}(w)\det(w,s,y)g_y(w)b_{zs}(w)w_2+\]
\[b_{ys}(w)b_{xs}(w)\det(w,s,z)g_z(w)w_3=0.\]
Note that $\det(w,s,x)=w_2s_3-s_2w_3$, $\det(w,s,y)=w_3s_1-w_1s_3$ and $\det(w,s,z)=w_1s_2-s_1w_2$. Substituting these into our last equation, we obtain
\[s_1w_2w_3b_{xs}(w)(b_{zs}(w)g_y(w)-b_{ys}(w)g_z(w))+\]
\[s_2w_1w_3b_{ys}(w)(b_{xs}(w)g_z(w)-b_{zs}(w)g_x(w))+\]
\[s_3w_1w_2b_{zs}(w)(b_{ys}(w)g_x(w)-b_{xs}(w)g_y(w))=0.\]
By \eqref{id1}
\[g_x(w)b_{ys}(w)-b_{xs}(w)g_y(w)=g_s(w)b_{yx}(w)\]
and similarly
\[g_z(w)b_{xs}(w)-b_{zs}(w)g_x(w)=g_s(w)b_{xz}(w),\]
\[g_y(w)b_{zs}(w)-b_{ys}(w)g_z(w)=g_s(w)b_{zy}(w).\]
Then
$$
s_1 b_{xs}(w)b_{yz}(w)w_2w_3+s_2 b_{ys}(w)b_{zx}(w)w_1w_3+s_3 b_{zs}(w)b_{xy}(w)w_1w_2=0.
$$

This implies that $V(\psi_{xyzs})$ contains $S'$, where $\psi_{xyzs}$ is defined as in the statement of the theorem. Observe that, by Lemma~\ref{segre}, the coefficient of $w_1^4w_2^2$ in the above equation is $2f_x(y)g_y(x)b_{zs}(x)$, so $\psi_{xyzs}(X) \not\equiv 0$.

To check that permuting the roles of $s$ and one of $\{ x,y,z\}$ does not change $V(\psi_{xyzs})$, we rewrite the above equation as 
$$
\det(s,y,z) b_{xs}(w)b_{yz}(w)\det(x,w,z)\det(x,y,w)
$$
$$
+\det(x,s,z) b_{ys}(w)b_{zx}(w)\det(w,y,z)\det(x,y,w)
$$
$$
+\det(x,y,s) b_{zs}(w)b_{xy}(w)\det(w,y,z)\det(x,w,z)=0.
$$
Switching $s$ and $x$ in the above, calculating with respect to the basis $\{ x,y,z\}$ and applying \eqref{id2}, i.e.
$$
b_{xs}(w)b_{yz}(w)+ b_{ys}(w)b_{zx}(w)+b_{zs}(w)b_{xy}(w)=0,
$$
we get the equation $s_1\psi_{xyzs}(w)=0$,
which implies that $V(\psi_{xyzs})$ does not change under any permutation of $\{ x,y,z,s\}$. 

\end{proof}

\begin{lemma} \label{curveeqn}
Let $s,x,y,z \in S'$. If $q$ is odd then for each $u \in \{x,y,z,s\}$, the point $u$ is a double point of $V(\psi_{xyzs})$ and the tangents to $V(\psi_{xyzs})$ at $u$ are the kernels of $f_u(X)$ and $g_u(X)$. 

Moreover, with respect to the basis $\{x,y,z\}$, the terms of $\psi_{xyzs}$, which are of degree at least four in one of the variables, are
$$
2s_2s_3 b_{zy}(x)f_x(X)g_x(X)\frac{g_s(x)}{g_x(s)} X_1^4+2s_1s_3 b_{xz}(y)f_y(X)g_y(X)\frac{g_s(y)}{g_y(s)} X_2^4$$
$$
+2s_1s_2 b_{yx}(z)f_z(X)g_z(X)\frac{g_s(z)}{g_z(s)} X_3^4.
$$
\end{lemma}

\begin{proof}
To calculate the terms divisible by $X_2^4$ observe that the degree of $X_2$ in $b_{yz}(X)$, $b_{ys}(X)$ and $b_{xy}(X)$ is one. Hence, the terms divisible by $X_2^4$ come from 
$$
s_1 b_{xs}(X)b_{yz}(X)X_2X_3+s_3 b_{zs}(X)b_{xy}(X)X_1X_2,
$$
and we have to take the coefficient of $X_2^2$ in $b_{xs}(X)$ and $b_{zs}(X)$, which is $b_{xs}(y)$ and $b_{zs}(y)$ respectively. Moreover, we have to take the coefficient of $X_2$ in $b_{yz}(X)$ and $b_{xy}(X)$, which is $f_y(X)g_z(y)-g_y(X)f_z(y)$ and $f_x(y)g_y(X)-g_x(y)f_y(X)$ respectively.

Putting this together we deduce that the terms divisible by $X_2^4$ are $X_2^4$ times
$$
s_1 X_3 b_{xs}(y)(f_y(X)g_z(y)-g_y(X)f_z(y))+s_3 X_1 b_{zs}(y)(f_x(y)g_y(X)-g_x(y)f_y(X)).
$$

Now, using Lemma~\ref{segre} and substituting $f_y(s)=f_y(x)s_1+f_y(z)s_3$, etc,
$$
\frac{s_3 b_{zs} (y)}{s_1 b_{xs}(y)}=\frac{s_3(f_z(y)g_s(y)-f_s(y)g_z(y))}{s_1(f_x(y)g_s(y)-g_x(y)f_s(y))}=\frac{s_3(f_z(y)g_y(s)+f_y(s)g_z(y))}{s_1(f_x(y)g_y(s)+g_x(y)f_y(s))}
$$

$$
=\frac{s_3(f_z(y)g_y(x)s_1+f_z(y)g_y(z)s_3+f_y(x)g_z(y)s_1+f_y(z)g_z(y)s_3)}{s_1(f_x(y)g_y(x)s_1+f_x(y)g_y(z)s_3+g_x(y)f_y(x)s_1+g_x(y)f_y(z)s_3)}
$$

$$
=\frac{s_3s_1(f_z(y)g_y(x)+f_y(x)g_z(y))}{s_1s_3(f_x(y)g_y(z)+g_x(y)f_y(z))}=-\frac{g_z(y)g_y(x)}{g_x(y)g_y(z)}.
$$

Therefore, the terms divisible by $X_2^4$ are $X_2^4$ times
$$
s_1 b_{xs}(y)(X_3(f_y(X)g_z(y)-g_y(X)f_z(y))-X_1\frac{g_z(y)g_y(x)}{g_x(y)g_y(z)}(f_x(y)g_y(X)-g_x(y)f_y(X)))
$$

$$
=s_1 b_{xs}(y)\frac{g_z(y)}{g_y(z)}(X_3(f_y(X)g_y(z)+g_y(X)f_y(z))-X_1\frac{g_y(x)}{g_x(y)}(f_x(y)g_y(X)-g_x(y)f_y(X)))
$$

$$
=s_1 b_{xs}(y)\frac{g_z(y)}{g_y(z)}(X_3f_y(X)g_y(z)+X_3g_y(X)f_y(z)+X_1f_y(x)g_y(X)+X_1g_y(x)f_y(X))
$$

$$
=2s_1b_{xs}(y)\frac{g_z(y)}{g_y(z)}f_y(X)g_y(X).
$$

By Lemma~\ref{segre}, 
$$
b_{xs}(y)=f_x(y)g_s(y)-g_x(y)f_s(y)=\frac{g_s(y)}{g_y(s)}(f_x(y)g_y(s)+g_x(y)f_y(s)),
$$
which gives
$$
b_{xs}(y)=\frac{g_s(y)}{g_y(s)}(f_x(y)(g_y(x)s_1+g_y(z)s_3)+g_x(y)(f_y(x)s_1+f_y(z)s_3))
$$
and so again by Lemma~\ref{segre}, 
$$
b_{xs}(y)=\frac{g_s(y)g_y(z)}{g_y(s)g_z(y)}s_3 b_{xz}(y).
$$
This last expression implies that the terms divisible by $X_2^4$ are $X_2^4$ times
$$
2s_1s_3 b_{xz}(y)f_y(X)g_y(X)\frac{g_s(y)}{g_y(s)}.
$$

By Lemma \ref{boxlemma} $b_{xz}(y)\neq 0$ and hence the multiplicity of $y$ is two.
Interchanging the roles of $y$ with either $x,z$ or $s$, we have that $u$ is a double point of $V(\psi_{xyzs})$ with tangents $f_u(X)$ and $g_u(X)$ for all $u \in \{ x,y,z,s \}$.
\end{proof}

\begin{lemma} \label{gcdlemma}
Let $U$ be a subspace of polynomials of ${\mathbb F}_q[X_1,X_2,X_3]$ and let $d$ denote the maximum degree of the elements of $U$. If $d \leqslant q$, then there exist $f,g\in U$ such that $\gcd(U)=\gcd(f,g)$. 
\end{lemma}

\begin{proof}
Let $b_1,\ldots,b_k$ be a basis for $U$ and denote $\gcd(U)$ by $\Gamma$. 
After dividing by $\Gamma$, we may assume that $\gcd(U)$ is 1. 
We may also assume $k>2$, since otherwise the statement is obvious. 

Suppose $k=3$. Let
$$
H=\{b_1+\alpha b_2 \ | \ \alpha\in {\mathbb F}_q \}\cup \{b_2\}.
$$
For any two different $h_1,h_2 \in H$ we have $\gcd(h_1,h_2)$ divides $\gcd(b_1,b_2)$. By assumption, $\gcd(\gcd(b_1,b_2),b_3)=1$. Thus, $\gcd(\gcd(h_1,b_3),\gcd(h_2,b_3))=1$. Therefore, if $h$ and $b_3$ are not coprime for each $h\in H$ then $b_3$ has at least $q+1$ distinct divisors, a contradiction since $\deg b_3 < q+1$. 

For $k>3$ the result follows by induction, since in $U'=\langle b_1,b_2,b_3\rangle_{{\mathbb F}_q}$ we can find $f,g$ such that $\gcd(f,g)=\gcd(U')$ 
and $\gcd(U)=\gcd(\gcd(U'),b_4,\ldots,b_k)=\gcd(f,g,b_4,\ldots,b_k)$. 
\end{proof}

Let $\Gamma(X)$ be the maximum common divisor of the polynomials in the subspace generated by
$$
\{ \psi_{xyzs}(X) \ | \ \{x,y,z,s\} \subset S' \}.
$$

\begin{lemma} \label{lotsofthem}
The set of points $V(\Gamma)$ contains at least $|S'|-c$ points of $S'$, for some constant $c$.
\end{lemma}

\begin{proof}
By Lemma~\ref{gcdlemma}, there are two polynomials $\psi(X),\psi'(X)$ in the subspace generated by 
$$
 \{ \psi_{xyzs}(X) \ | \ \{x,y,z,s\} \subset S' \}
$$ 
whose maximum common divisor is $\Gamma$. If $\psi=\phi\Gamma$ and $\psi'=\phi'\Gamma$ then, by B\'ezout's theorem, $V(\phi) \cap V(\phi')$ contains a constant number of points. The set $S' \setminus (V(\phi) \cap V(\phi'))$ is contained in $V(\Gamma)$.
\end{proof}

\begin{lemma}\label{singular=double}
	The singular points of $V(\Gamma)$ are double points.
\end{lemma}

\begin{proof}
	It follows from the definition of $\Gamma$ and from Lemma \ref{curveeqn}.
\end{proof}

\begin{lemma} \label{notdoubles}
The variety $V(\Gamma)$ does not have more than a constant number of double points at the points of $S'$.
\end{lemma}

\begin{proof}
Suppose that $V(\Gamma)$ has more than a constant number of double points. Then, since $\Gamma$ has degree at most six and any of its derivatives have degree at most five, B\'ezout's theorem implies that $\Gamma$ is reducible. Each irreducible subvariety of $V(\Gamma)$ has a finite number of double points. Distinct irreducible subvarieties of $\Gamma$ contain a constant number of points of $S'$ in their intersection, by B\'ezout's theorem. Therefore, $V(\Gamma)$ must contain repeated subvarieties in its decomposition into irreducible subvarieties containing more than a constant number of points of $S'$.

If $V(\Gamma)$ has a repeated subvariety, say $G^2 \mid \Gamma$ and $G$ vanishes at $x \in S'$, then $V(\psi_{xyzs})$ has a point of multiplicity at least three at $x$, since it has distinct tangents and for any tangent to $V(G)$ there corresponds a repeated tangent to $V(\psi_{xyzs})$. This would contradict Lemma \ref{curveeqn}.

Therefore $V(\Gamma)$ has a constant number of points of $S'$ in its repeated subvarieties, a contradiction.
\end{proof}

\begin{lemma} \label{degree6}
If $q$ is large enough, then the polynomial $\Gamma(X)$ does not have degree six.
\end{lemma}

\begin{proof}
If $\Gamma(X)$ is of degree six then it is a constant multiple of $\psi_{xyzs}$, for all subsets $\{x,y,z,s\}$ of $S'$. Lemma~\ref{curveeqn} implies that $V(\Gamma)$ has a double point at every point of $S'$, contradicting Lemma~\ref{notdoubles}.
\end{proof}

\begin{lemma} \label{degree5}
If $q$ is large enough, then the polynomial $\Gamma(X)$ does not have degree five.
\end{lemma}

\begin{proof}
Suppose that $u,v \in S'$ are simple points of $V(\Gamma)$ and that $\Gamma$ is of degree five. Then $V(\psi_{xyuv})$ contains a line, for all subsets $\{x,y,u,v\}$ of $S'\cap V(\Gamma)$ and this line is zero at $u$ and $v$, since $V(\psi_{xyuv})$ has double points at $u$ and $v$ and $V(\Gamma)$ does not. But then $V(\Gamma)$ has double points at $x$ and $y$, so all other points of $S' \cap V(\Gamma)$, contradicting Lemma~\ref{notdoubles}.
\end{proof}

\begin{lemma} \label{degree4}
If $q$ is large enough, then the polynomial $\Gamma(X)$ does not have degree four.
\end{lemma}

\begin{proof}
Define $\phi_{xyzs}(X)$ by $\psi_{xyzs}(X)=\Gamma(X)\phi_{xyzs}(X)$. 

Let $D$ be the set of double points of $V(\Gamma)$. By Lemma~\ref{notdoubles}, $|D|$ is constant. 

Suppose that there is an $x,y,z \in S' \setminus D$ for which $V(\phi_{xyzs})$ is a non-degenerate conic containing $x,y,z$ and $s$, for more than a constant number of $s \in S'$. The tangents at $u \in \{x,y,z\}$ to $V(\psi_{xyzs})$ are the kernels of $f_u(X)$ and $g_u(X)$. There are 8 possible conics which are zero at $x,y,z$ and have tangents either $f_u(X)$ or $g_u(X)$ at $u \in \{x,y,z \}$. Hence, there is a conic $V(\phi)$ and a subset $S''$ of $S'$, such that $V(\phi_{xyzs})=V(\phi)$ for all  $x,y,z,s \in S''$, and where $S''$ consists of more than a constant number of points of $S'$. By Lemma~\ref{lotsofthem}, $V(\Gamma)$ contains $|S'|-c$ points of $S'$, so B\'ezout's theorem implies that $V(\phi)$ is a subvariety of $V(\Gamma)$. But then, for $s \in S''$,  $V(\psi_{xyzs})$ has double points at $x,y,z,s$ with repeated tangents. By Lemma~\ref{curveeqn}, $x$ is a double point of $V(\psi_{xyzs})$ with tangents $f_x(w)$ and $g_x(w)$, and similarly for $y$, $z$ and $s$, a contradiction.

Therefore, for all $x,y,z \in S' \setminus D$, the variety $V(\phi_{xyzs})$ is a degenerate conic for more than a constant number of points $s \in S' \setminus  D$. The two lines in $V(\phi_{xyzs})$ are the tangents to the curve $V(\psi_{xyzs})$. Since these two lines contain the four points $x,y,z,s$ one of them must be the line defined by $g_x$, $g_y$, or $g_z$, contradicting the fact that the points of $S'$ are joined by lines which are not $3$-secants to $S$.
\end{proof}

\begin{lemma} \label{degree3}
If $q$ is large enough, then the polynomial $\Gamma(X)$ is not irreducible of degree three.
\end{lemma}

\begin{proof}
Let $D$ be the set of double points of $V(\Gamma)$ and let $x,y,z \in S' \setminus  D$.

Let $\phi_{xyzs}(X)$ be defined by $\psi_{xyzs}(X)=\Gamma(X)\phi_{xyzs}(X)$.

By Lemma~\ref{curveeqn}, $x$ is a double point of $V(\psi_{xyzs})$ with tangents $f_x(w)$ and $g_x(w)$, and similarly for $y$, $z$ and $s$.

Since $V(\Gamma)$ has degree three and simple zeros at $x$, $y$ and $z$, we have that
$$
\Gamma(X)=c_1X_1^2h_x^+(X)+c_2X_2^2h_y^+(X)+c_3X_3^2h_z^+(X)+c_4X_1X_2X_3.
$$
Furthermore, we have that $h_x^+(X)$ is either $f_x(X)$ or $g_x(X)$ and similarly for $h_y^+(X)$ and $h_z^+(X)$.

By Lemma~\ref{curveeqn}, the terms of $\psi_{xyzs}(X)$ which are of degree 4 in one of the variables are 
$$
2s_2s_3 b_{zy}(x)f_x(X)g_x(X)\frac{g_s(x)}{g_x(s)} X_1^4+2s_1s_3 b_{xz}(y)f_y(X)g_y(X)\frac{g_s(y)}{g_y(s)} X_2^4$$
$$
+2s_1s_2 b_{yx}(z)f_z(X)g_z(X)\frac{g_s(z)}{g_z(s)} X_3^4.
$$
Hence, we have that for some $d(s)$,
$$
\phi_{xyzs}(X)=2c_1^{-1}b_{zy}(x)s_2s_3\frac{g_s(x)}{g_x(s)}X_1^2h_x^-(X)+2c_2^{-1}b_{xz}(y)s_1s_3\frac{g_s(y)}{g_y(s)}X_2^2h_y^-(X)
$$
$$
+2c_3^{-1}b_{yx}(z)s_1s_2\frac{g_s(z)}{g_z(s)}X_3^2h_z^-(X)+d(s)X_1X_2X_3,
$$
where $h_x^+h_x^-=f_xg_x$, etc.

The coefficient of $X_1X_2$ in $b_{zs}(X)=f_z(X)g_s(X)-g_z(X)f_s(X)$ is 
$$
f_z(x)g_s(y)-g_z(x)f_s(y)+f_z(y)g_s(x)-g_z(y)f_s(x).
$$

Therefore, by Lemma~\ref{curve} and Lemma~\ref{segre}, the coefficient of $X_1^3X_2^3$ of $\psi_{xyzs}(X)$ is
$$
2s_3f_x(y)g_y(x)(f_z(x)g_s(y)-g_z(x)f_s(y)+f_z(y)g_s(x)-g_z(y)f_s(x)).
$$
The coefficient of $X_1^3X_2^3$ in $\Gamma(X)\phi_{xyzs}(X)$ is
$$
2s_3\Big(b_1s_1\frac{g_s(y)}{g_y(s)}+b_2s_2\frac{g_s(x)}{g_x(s)}\Big),
$$
for some $b_i$ dependent only on $x,y,z$.
Equating these two expressions and applying Lemma~\ref{segre}, we have
$$
\frac{g_s(y)}{g_y(s)}(f_x(y)g_y(x)(f_z(x)g_y(s)+g_z(x)f_y(s))-b_1s_1)
$$
$$
=-\frac{g_s(x)}{g_x(s)}(f_x(y)g_y(x)(f_z(y)g_x(s)+g_z(y)f_x(s))-b_2s_2).
$$
By Lemma~\ref{gfunc},
$$
\frac{g_s(y)}{g_y(s)}b_{yz}(s)s_2=-\frac{g_s(x)}{g_x(s)}b_{xz}(s)s_1.
$$

Thus, combining these last two equations we have that
$$
\rho_{12}(s)=s_1b_{xz}(s)(f_x(y)g_y(x)(f_z(x)g_y(s)+g_z(x)f_y(s))-b_1s_1)
$$
$$
-s_2b_{yz}(s)(f_x(y)g_y(x)(f_z(y)g_x(s)+g_z(y)f_x(s)-b_2s_2)=0.
$$
The coefficient of $X_1^2X_3^2$ in $\rho_{12}(X)$ is 
$$
2f_x(z)g_z(x)f_x(y)g_y(x)(f_z(x)g_y(z)+g_z(x)f_y(z)).
$$ 
By Lemma~\ref{segre},
$$
f_z(x)g_y(z)+g_z(x)f_y(z)=(f_x(z)g_y(z)-g_x(z)f_y(z))\frac{f_z(x)}{f_x(z)}=b_{xy}(z)\frac{f_z(x)}{f_x(z)}.$$

By Lemma~\ref{boxlemma}, we have that $\rho_{12} \neq 0$. Thus, we get a curve $V(\rho_{12})$ of degree four which contains at least $|S'|-c$ points of $S'$.

Since $V(\Gamma)$ is an irreducible curve of degree three containing at least $|S'|-c$ points of $S'$, B\'ezout's theorem and Lemma~\ref{lotsofthem} imply that $\rho_{12}(X)$ is a multiple of $\Gamma(X)$, i.e. $\rho_{12}(X)=\alpha_{12}(X) \Gamma(X)$, for some linear form $\alpha_{12}(X)$.

The terms divisible by $X_1^3$ in $\rho_{12}(X)$ are $a_1(g_z(x)f_x(X)-f_z(x)g_x(X))$ for some $a_1$ dependent only on $x,y,z$.

The terms divisible by $X_1^2$ in $\Gamma$ are $c_1h_x^+(X)$. Therefore, in the product $\alpha_{12}(X)\Gamma(X)$, the terms divisible by $X_1^3$ are a multiple of $h_x^+(X)$. But $h_x^+(X)$ is either $f_x(X)$ or $g_x(X)$, which implies $f_x(X)$ and $g_x(X)$ are multiples of each other, which they are not, or $a_1=0$.

Hence, $a_1=0$ and similarly $a_2=0$, which implies that $\rho_{12}(X)=a_3 X_3 \Gamma (X)$, for some $a_3 \in {\mathbb F}_q$.

Note that $a_1=0$ and $a_2=0$ imply that 
$$
\rho_{12}(X)=f_x(y)g_y(x)X_3((f_z(x)g_y(z)+g_z(x)f_y(z))b_{xz}(X)X_1
$$
$$
-(f_z(y)g_x(z)+g_z(y)f_x(z))b_{yz}(X)X_2).
$$
The terms divisible by $X_1^2X_3$ in $\rho_{12}(X)$ are
$$
f_x(y)g_y(x)(f_z(x)g_y(z)+g_z(x)f_y(z))(g_z(x)f_x(X)-f_z(x)g_x(X)).
$$
The terms divisible by $X_1^2$ in $\Gamma$ are $c_1h_x^+(X)$, so again we have that 
$f_x(X)$ and $g_x(X)$ are multiples of each other, which they are not, or $f_z(x)g_y(z)+g_z(x)f_y(z)=0$. The last equation and Lemma~\ref{segre} imply that $b_{xy}(z)=0$, contradicting Lemma~\ref{boxlemma}.
\end{proof}

\begin{lemma} \label{gsconcur}
Suppose that $x,y,z$ are points of a $S_{4/3}$ which lie on a conic $C$ whose tangent at $u \in \{x,y,z \}$ is the kernel of $f_u(X)$.  Then the lines $\ker g_x$, $\ker g_y$ and $\ker g_z$ are concurrent with a point. Furthermore, if the line joining $x$ and $y$ is $\ker g_x$ ($=\ker g_y$) then the line $\ker g_z$ is incident with $S\cap \ker g_x \setminus \{x,y \}$.
\end{lemma}

\begin{proof}
If the line joining $x$ and $y$ is not $\ker g_x$ then, as in the proof of Lemma~\ref{segre}, we have that
$$
g_x(y)f_x(z)g_y(z)f_y(x)g_z(x)f_z(y)=-g_x(z)f_x(y)g_y(x)f_y(z)g_z(y)f_z(x).
$$
Again, as in the proof of Lemma~\ref{segre} but using the conic $C$ in place of the $S$, we have that
$$
f_x(z)f_y(x)f_z(y)=f_x(y)f_y(z)f_z(x),
$$
from which we deduce that
$$
g_x(z)g_y(x)g_z(y)=-g_x(y)g_y(z)g_z(x).
$$
This is precisely the condition that the point $(g_y(z)g_x(y),g_x(z)g_y(x),-g_x(y)g_y(x))$ is incident with $\ker g_x$, $\ker g_y$ and $\ker g_z$.

If the line joining $x$ and $y$ is $\ker g_x$ then let $(1,a,0)$ be the third point of $S$ on the line $\ker g_x$, coordinates with respect to the basis $\{x,y,z \}$. As in the proof of Lemma~\ref{segre}, using $S \setminus \{(1,a,0) \}$ instead of $S$, we have that
$$
f_x(z)f_y(x)g_z(x)f_z(y)=-af_x(y)f_y(z)g_z(y)f_z(x).
$$
As before, we have that
$$
f_x(z)f_y(x)f_z(y)=f_x(y)f_y(z)f_z(x).
$$
Therefore, $g_z(x)=-ag_z(y)$, which is precisely the condition that the point $(1,a,0)$ is incident with $\ker g_z=\ker (g_z(x)X_1+g_z(y)X_2)$.

\end{proof}

\begin{proof} (of Theorem~\ref{main})
It is enough to prove the result for $q$ large enough, say $q>q_0$, such that all of our previous lemmas hold, since for $q<q_0$ the number of odd-secants is clearly less than $2q-c$ with $c=2q_0$.

The 3-secants to $S$ partition $S_{4/3}$ into at least $\frac{1}{3}|S_{4/3}|$ parts, where two points are in the same part if and only if they are joined by a $3$-secant to $S$. Let $S'$ be a subset of $S_{4/3}$ of at least $\frac{1}{3}|S_{4/3}|$ points, which does not contain two points joined by a $3$-secant to $S$, so no two points from the same part of the partition.

By Lemma~\ref{degree6}, Lemma~\ref{degree5}, Lemma~\ref{degree4} and Lemma~\ref{degree3}, all but at most $c'$ points of $S'$ are contained in a conic $C \subseteq V(\Gamma)$, for some constant $c'$. Since $C \subseteq V(\Gamma)$ and $V(\Gamma) \subseteq V(\psi_{xyzs})$, Lemma~\ref{curveeqn} implies that the tangent to the conic at the point $x \in S' \cap  C$ is either the kernel of $f_x$ or the kernel of $g_x$. 

Let $X$ be the subset of $S_{4/3}$ containing all the points of $S_{4/3} \setminus C$.

Redefine $S'$ to be a subset of $S_{4/3}$ taking a point $x \in X$ from each part of the partition of $S_{4/3}$ by the $3$-secants and at least $c'+5$ points from $S \cap C$. Note that $|X| \leqslant 3|S' \cap X|$. 
Applying the same lemmas to this re-defined $S'$, we have that all but at most $c'$ points of $S'$ are contained in a conic $C'$. But $C$ and $C'$ share at least five points, so are the same. Therefore, all but at most $c'$ points of $S'$ are contained in $C$. However, the points of $S' \cap X$ are not in $C$, so we have $|S' \cap X| \leqslant c'$ and so $|X| \leqslant 3c'$. Therefore, all but at most $3c'$ points of $S_{4/3}$ are contained in the conic $C$.

Recall that, for any $x \in S$, the {\em weight} of $x$ is
$$
w(x) = \sum \frac{1}{| \ell \cap S|},
$$
where the sum is over the lines $\ell$ incident with $x$ and an odd number of points of $S$.

By Lemma~\ref{szero}, there are at most two points of weight zero.

Observe that the number of odd secants is $\sum_{x \in S} w(x).$

Let $T$ be the subset of $C \cap S_{4/3}$ such that the tangent to $C$ at the point $x$ is $\ker g_x$. Let $T'$ be the set of points of $S$ which are joined to a point of $T$ by a $3$-secant. For all $y \in T' \setminus S_{4/3}$, the weight $w(y) \geqslant \frac{8}{3}$. And note that $|T' \cap S_{4/3}| \leqslant 3c'$, since $|S_{4/3} \setminus C | \leqslant 3c'$. If $y \in T' \setminus S_{4/3}$ then $y$ is incident with at most two lines $\ker g_x$, where $x \in T$, since $\ker g_x$ is a tangent to $C$ for $x \in T$. Hence, counting pairs $(y, \ell)$ where $y \in  T'$ is incident with $\ell$, a $3$-secant incident with a point of $T$, we have 
$$
2|T' \setminus S_{4/3}|+3c' \geqslant 2|T|.
$$

By Lemma~\ref{gsconcur}, there is a point $m$, which is incident with $\ker g_x$, for all $x \in (C \cap S_{4/3}) \setminus T$. 

Suppose $m \in S$. Then the weight of $m$ satisfies
$$
w(m) \geqslant \tfrac{1}{3} \tfrac{1}{2} |(C \cap S_{4/3}) \setminus T|+ \tfrac{1}{2} |(C \cap S_{4/3}) \setminus T|=\tfrac{2}{3}|(C \cap S_{4/3}) \setminus T|.
$$
The weight of a point in $S \setminus S_{4/3}$ is at least two, so we have that the number of odd secants 
$$
\sum_{x \in S} w(x) \geqslant 2(|S|-|S_0|-|S_{4/3}|-|T'|-1) + \tfrac{4}{3}|S_{4/3}|+\tfrac{2}{3}|(C \cap S_{4/3}) \setminus T|+\tfrac{8}{3} |T' \setminus S_{4/3}| \geqslant 2q-c'',
$$
for some constant $c''$.

Suppose $m \not\in S$. Then, by Lemma~\ref{gsconcur}, there are no pair of points of $(C \cap S_{4/3}) \setminus T$ collinear with $m$. Therefore, for $x \in (C \cap S_{4/3}) \setminus T$ the two points of $(S \cap \ker g_x)\setminus \{ x\}$ are of weight at least $\frac{8}{3}$ or are in $X$. 

Let $U$ be the set of points of $S \setminus S_{4/3}$ which are incident with $\ker g_x$ for $x \in C \cap S_{4/3}$. For a fixed $u \in U$, there are at most two $3$-secants which are incident with a point $x \in T$, since the $3$-secants are tangents to the conic $C$ at the points of $T$, and at most one $3$-secant which is incident with an $x \in (C \cap S_{4/3}) \setminus T$. Hence, each $u\in U$ is incident with at most three $3$-secants, incident with a point $x \in C \cap S_{4/3}$. Let $U_i$ denote the points $u$ which are incident wtih  $i$ $3$-secants, incident with a point $x \in C \cap S_{4/3}$. 

Counting pairs $(u, \ell)$ where $u \in U\cup X$ is a point incident with $\ell$, a $3$-secant incident with a point $x \in C \cap S_{4/3}$, we have
$$
2|C \cap S_{4/3}|=c'''+U_1+2U_2+3U_3,
$$
for some constant $c'''$, since $|X|$ is constant.

Observe that $u \in U_1$ does not have weight $4/3$, so its weight is at least $8/3$.

The number of odd secants is at least 
$$
\sum_{x \in S} w(x) \geqslant 2(|S|-|S_0|-|S_{4/3}|-|U_1|-|U_2|-|U_3|) + \tfrac{4}{3}|S_{4/3}|+\tfrac{8}{3}(|U_1|+|U_2|)+4 |U_3|
$$
$$
\geqslant 2q+\tfrac{2}{3}(|S_{4/3}|-|U_2|)-c'''' \geqslant 2q-c,
$$
for some constants $c'''$ and $c$.

\end{proof}

\section{Larger sets of points with few odd secants.}

Most of the ideas used for sets of size $q+2$ carry through to sets of larger size. However, the degree of the curve increases and we are unsure if this is of any use or not. We briefly describe how one might proceed with larger sets, supressing much of the detail.

Let $S$ be a set of $q+t+1$ points in $\mathrm{PG}(2,q)$, $q$ odd. For any $x \in S$, the {\em weight} of $x$ is
$$
w(x) = \sum \frac{1}{| \ell \cap S|},
$$
where the sum is over the lines $\ell$ incident with $x$ and an odd number of points of $S$.

Let $S_{t}$ be the set of points of $S$ which are incident with at most one tangent.

\begin{lemma} \label{notconstant3}
If $|S_{t}|$ is less than some constant then there are constants $c$ and $c'$ (depending on $t$), such that the number of odd secants to $S$ is at least $(2+\frac{1}{t+3})q-c$ if $t$ is odd and at least $2q-c'$ if $t$ is even.
\end{lemma}

\begin{proof}
The number of odd secants is $\sum_{x \in S} w(x)$. Put $c''=|S_t|$.  

Suppose $t$ is odd. If $x \in S \setminus S_t$ then $w(x) \geqslant 2+\frac{1}{t+3}$, since $w(x)$ is minimised when $x$ is incident with two tangents, one $(t+3)$-secant. Then 
\[\sum_{x \in S} w(x) \geq |S\setminus S_t|\left(2+\frac{1}{t+3}\right) = q\left(2+\frac{1}{t+3}\right)-(c''-t-1)\left(2+\frac{1}{t+3}\right).\]

On the other hand, suppose $t$ is even. If $x \in S \setminus S_t$ then $w(x) \geqslant 2$, since $w(x)$ is minimised when $x$ is incident with two tangents. Then 
\[\sum_{x \in S} w(x) \geq |S\setminus S_t|2= 2q-2(c''-t-1).\]

\end{proof}

We can prove that $S_t$ has some algebraic structure when $t$ is small. 

For each $x \in S_t$, let $f_x$ be a linear form whose kernel is the tangent to $S$ at $x$ (if there is no tangent to $S_t$ at $x$ then let $f_x$ be identically $1$). Let $g_x$ be the product of $t$ linear forms (or $t-1$ linear forms, if there is no tangent to $S_t$ at $x$) whose kernels correspond to the $i$-secants, where $i \geqslant 3$, counted with multiplicity, so that an $i$-secant is counted $i-2$ times.

Let $S'$ be a subset of $S_t$ with the property that two points of $S'$ are joined by a bi-secant to $S$. Fix an element $e \in S'$ and scale $f_x$ so that $f_x(e)=f_e(x)$. Similarly, scale $g_x$ so that $g_x(e)=g_e(x)$. 

As in Lemma~\ref{segre} we have the following lemma.

\begin{lemma} \label{segret}
For all $x,y \in S'$,
$$
f_x(y)g_y(x)=(-1)^{t}f_y(x)g_x(y).
$$
\end{lemma}

Interpolating $g_x(X)$ as in Lemma 3.1 from \cite{Ball2012}, one can show that the points of $S'$ are contained in a curve of degree $t^2+5t+1$ and this can be improved to a curve of degree $12$ for $t=2$.

\bibliographystyle{plain}

  \bigskip
  
   Simeon Ball,\\
   Departament de Matem\`atiques, \\
Universitat Polit\`ecnica de Catalunya, \\
M\`odul C3, Campus Nord,\\
c/ Jordi Girona 1-3,\\
08034 Barcelona, Spain \\
   {\tt simeon@ma4.upc.edu}
  
  \bigskip
  
Bence Csajb\'ok,\\
MTA--ELTE Geometric and Algebraic Combinatorics Research Group,\\
ELTE E\"otv\"os Lor\'and University, Budapest, Hungary\\
Department of Geometry\\
1117 Budapest, P\'azm\'any P.\ stny.\ 1/C, Hungary\\
{\tt csajbokb@cs.elte.hu}

\end{document}